\documentclass[12pt]{amsart}
\usepackage[english]{babel}
\usepackage{amsfonts,amssymb,latexsym,amscd}

\theoremstyle{plain}
\newtheorem{theorem}{Theorem}

\newtheorem{proposition}{Proposition}
\newtheorem{corollary}{Corollary}
\theoremstyle{definition}
\newtheorem{definition}{Definition}
\theoremstyle{remark}
\newtheorem{remark}{Remark}
\newtheorem{example}{Example}

\usepackage{hyperref}

\begin{document}

\title{Strong Movable Categories and Strong Movability of Topological Spaces}

\author{Pavel S. Gevorgyan}
\address{Moscow Pedagogical State University}
\email{pgev@yandex.ru}

\author{T. A. Avakyan}
\address{Moscow Power Engineering Institute, Moscow, Russia}

\begin{abstract}
The paper is devoted to one of the important notions of the shape theory: that of strong movability, which was primarily introduced by K. Borsuk for metrizable compacts. A strong movability criterion is proved for topological spaces, which in particular reveals a new, categorical approach to the strong movability.
\end{abstract}



\maketitle

\section{Introduction}

Strong movability is one of the important notions of the shape theory. It had a significant role first of all in the study of stable topological spaces, i.e. spaces having a shape of some CW-complex (see \cite{borsuk}, \cite{dyd}, \cite{wat}). The strong movability notion primarily was introduced by K. Borsuk \cite{borsuk} for metrizable compacts by means of the neighborhoods of the given compact, which is embedded with its closure in some AR-space. For arbitrary topological space, the notion of strong movability was introduced by Marde\v{s}i\'{c} \cite{mardsegal2} by means of the inverse spectrums associated with the given space.

In particular, the results of this paper permit to apply a new, more general approach to the notion of strong movability. This approach is based on some ideas applied in \cite{gev1}, \cite{gev2} to the movability of topological spaces. As it shown in \cite{gevpop} and \cite{pop}, the same ideas are applicable also for studying the uniform movability of topological spaces.

The criterion of strong movability of topological spaces, which is obtained in the present paper, leads also to some improvements and remarks on the statements of Theorem 4.1 in \cite{gev1} and Theorem 3 in \cite{gev2}.

\section{Main Notions}

Recall some definitions and notions of the category and functor theory and shape theory, which are to be used in this paper.

The class of all objects of a category $K$ is denoted by $\mathrm{Ob}(K)$. If $A$ and $B$ are objects of $K$, for briefness, $A,B \in K$, then the class of all morphisms from $A$ to $B$ is denoted by $\mathrm{Hom}_K(A,B)$ or, briefly, $\mathrm{Hom}(A,B)$, when the context clearly shows which category is considered. We denote this class also by $\mathrm{Mor}_K(A,B)$.

\begin{definition}
A category $K$ is said to be functorially dominated by a category $L$, or briefly $K\leqslant L$, if there exist functors $F:K\to L$ and $G:L\to K$, such that $G\circ F = 1_K$.
\end{definition}

\begin{definition}\label{def-fm}
Let $F$ and $G$ be covariant functors from the category $K$ to the category $L$. Then, a functor morphism or a natural transformation $\varphi:F\to G$ of $F$ to $G$ is said to be given, if for any object $A\in K$ a morphism $\varphi(A):F(A)\to G(A)$ of the category L is given, such that for any morphism $f:A\to B$ of the category $K$ the diagram
\begin{center}
\begin{picture}(80,90)
\put(0,0){$F(B)$} \put(0,50){$F(A)$} \put(80,0){$G(B)$}
\put(80,50){$G(A)$}
\put(30,54){\vector(1,0){43}} \put(13,45){\vector(0,-1){28}}
\put(30,4){\vector(1,0){43}} \put(93,45){\vector(0,-1){28}}
\put(-13,25){$F(f)$} \put(97,25){$G(f)$} \put(40,-8){$\varphi(B)$}
\put(40,58){$\varphi(A)$} 
\end{picture} 
\end{center} \ \\
is commutative. For indicating that there is a natural transformation $\varphi:F\to G$ of the functor $F$ to the functor $G$, we write $F\rightsquigarrow G$.
\end{definition}

Let $(K_i)$, $i\in I$, be a family of categories. The product of categories $K_i$, $i\in I$, is the category $\prod\limits_{i\in I}K_i$ the objects of which are all possible families $(X_i)$, $i\in I$, where $X_i$ are the objects of categories $K_i$, $i\in I$, and the morphisms are defined as
$$\mathrm{Hom}_{\prod\limits_{i\in I}K_i}((X_i),(Y_i)) = \prod_{i\in I}
\mathrm{Hom}_{K_i}(X_i,Y_i) .$$
If 
$$(u_i)\in \mathrm{Hom}_{\prod\limits_{i\in I}K_i}((X_i),(Y_i))$$
and
$$(v_i)\in \mathrm{Hom}_{\prod\limits_{i\in I}K_i}((Y_i),(Z_i)),$$
then the composition of these morphisms is defined by the formula
$$(v_i)\circ (u_i) = (v_i\circ u_i) .$$

By $HTOP$ we mean a homotopy category of topological spaces and by $HCW$ its complete subcategory, the objects of which are topological spaces that have the homotopy type of CW-complexes.

\begin{definition}[K. Morita \cite{morita}]\label{defass}
The inverse system or the inverse spectrum $\{X_\alpha , p_{\alpha \alpha '}, A\}$ of a category $HCW$ is said to be \emph{associated} with the topological space $X$, if for any $\alpha \in A$ there exists a homotopy class $p_\alpha : X \to X_\alpha $ such that:

$1^\circ$ \ $p_{\alpha \alpha'} \circ p_{\alpha'}=p_\alpha$ for any $\alpha, \alpha' \in A$, $\alpha<\alpha'$; 

$2^\circ$ for any homotopy class $f:X\to Q$, $Q\in$ Ob($HCW$), there are some $\alpha \in A$ and a homotopy class $f_\alpha:X_\alpha \to Q$ such that $f=f_\alpha \circ p_\alpha$;

$3^\circ$ for any $\alpha \in A$ and any homotopy classes $f_\alpha, g_\alpha :X_\alpha \to Q$, $Q\in$Ob($HCW$), which satisfy the condition $f_\alpha \circ p_\alpha = g_\alpha \circ p_\alpha$, there exists an index $\alpha ' \in A$, $\alpha \leqslant \alpha '$, such that $f_\alpha \circ p_{\alpha \alpha '} = g_\alpha \circ p_{\alpha \alpha '}$.
\end{definition}

It is known (see \cite{morita}) that there is an associated inverse spectrum to any topological space $X$ of the category $HCW$. In other words, $HCW$ is a dense subcategory of the category $HTOP$. Hence, the shape theory can be constructed for all topological spaces by application of the associated inverse spectrums.

The shape theory of topological spaces can be constructed also by means of comma categories (see \cite{mardsegal1}). To this end, to each topological space $X$ a comma category $W^X$ is put into correspondence, whose objects are homotopy classes $f : X \to Q$ and morphisms are commutative triangles of the form
\begin{center}
\begin{picture}(80,80)
\put(0,0){$Q'$} \put(40,40){$X$} \put(80,0){$Q$ ,}
\put(38,38){\vector(-1,-1){28}} \put(52,38){\vector(1,-1){28}}
\put(74,3){\vector(-1,0){60}} \put(14,21){$f'$} \put(73,21){$f$}
\put(40,-4){$\eta $}
\end{picture}   
\end{center}\ \\
where $Q, Q' \in$ Ob($HCW$). Then, the shape morphism $F:X\to Y$ is determined as the covariant
functor $F:W^Y \to W^X$ which does not change the homotopy class $\eta$.

\begin{definition}[S. Marde\v{s}i\'{c}, J. Segal \cite{mardsegal1}]\label{def1}
The inverse system $\{X_\alpha , p_{\alpha \alpha '}, A\}$ of a category $HCW$ is called \emph{strongly movable} if the following condition is fulfilled:

(SM1) For any  $\alpha \in A$, there exists $\alpha ' \in A$, $\alpha ' \geqslant \alpha $, such that for any $\alpha '' \in A$, $\alpha '' \geqslant \alpha $, there exists $\alpha ^* \in A$, $\alpha ^* \geqslant \alpha '$, $\alpha ^* \geqslant \alpha ''$ and a homotopy class $r^{\alpha ' \alpha ''}:X_{\alpha '}\to X_{\alpha ''}$ such that the below equalities simultaneously hold:
\begin{equation}
p_{\alpha \alpha'}=p_{\alpha \alpha''}r^{\alpha ' \alpha ''} ,
\end{equation}
\begin{equation}\label{eq1}
r^{\alpha ' \alpha ''} \circ p_{\alpha' \alpha ^*}=p_{\alpha'' \alpha ^*}.
\end{equation}
A topological space $X$ is said to be \emph{strongly movable}, if in category $HCW$ there exists an
associated with it strongly movable inverse system $\{X_\alpha , p_{\alpha \alpha '}, A\}$.
\end{definition}

For more detailed information on the above notions of category and functor theory and shape theory, we recommend the monographs \cite{bukur} and \cite{mardsegal1}.

\section{Strongly Movable Categories}

Let $K$ and $L$ be any categories and let $\Phi:K\to L$ be any covariant functor.

\begin{definition}\label{movcat}
We say that the category $K$ is movable with respect to the category $L$ and the functor  $\Phi:K\to L$, if for any object $X\in ob(K)$ there exist another object $M(X)\in ob(K)$ and a morphism $m_X\in Mor_{K}(M(X), X)$, such that for any object $Y\in ob(K)$ and any morphism $p\in Mor_{K}(Y, X)$ there is another morphism $u\in Mor_L(\Phi(M(X)), \Phi(Y))$ for which  $\Phi(p)\circ u = \Phi(m_X)$.
\end{definition}

If $K$ is a complete subcategory of the category $L$ and $\Phi:K\hookrightarrow L$ is an embedding functor of Definition \ref{movcat}, then the subcategory $K$ is said to be \emph{movable with respect to the category L} (see Definition 1 in \cite{gev1}).

Supposing that $X$ is a topological space with a comma category  $W^X$, we introduce a covariant functor $\Omega:W^X\to HCW$, which puts into correspondence to each object $f:X\to Q$ an object $Q\in HCW$ and to each morphism $\eta : (f:X\to Q)\to (f':X\to Q')$, $(\eta \circ f = f')$ a morphism  $\eta : Q\to Q'$. We call $\Omega$ \emph{forgetful functor}.

\begin{theorem}\label{th-1}
A topological space $X$ is movable if and only if the comma category $W^X$ is movable with respect to the category $HCW$ and the forgetful functor $\Omega:W^X\to HCW$.
\end{theorem}

The proof of this theorem follows from Theorem 4 in \cite{gev1} (see also Theorem 4.2 in \cite{gev2}). Note that the above Theorem \ref{th-1} corrects an inaccuracy which we found in the statement of Theorem 3 in \cite{gev1} (see also Theorem 4.1 in \cite{gev2}).

\begin{proposition}\label{prop1}
Let a category $K$ be movable with respect to the category $L$ and the functor $\Phi:K\to L$, and let $F:L\to L'$ be an arbitrary functor. Then the category $K$ is movable with respect to the category $L'$ and the functor $F\circ \Phi:K\to L'$.
\end{proposition}

\begin{proof}
By Definition \ref{movcat}, for any object $X\in ob(K)$ there exist another object $M(X)\in ob(K)$ and
a morphism $m_X\in Mor_{K}(M(X), X)$, such that for any object $Y\in ob(K)$ and any morphism $p\in Mor_{K}(Y, X)$ there is a morphism $u\in Mor_L(\Phi(M(X)), \Phi(Y))$ for which $\Phi(p)\circ u = \Phi(m_X)$. Hence, $F(\Phi(p))\circ F(u) = F(\Phi(m_X))$, since $F:L\to L'$ is a covariant functor. This means that $F(u)\in Mor_{L'}((F\circ \Phi)(M(X)), (F\circ \Phi)(Y))$ is the required morphism, i.e. the category $K$ is movable with respect to the category $L'$ and the functor $F\circ \Phi:K\to L'$.
\end{proof}

\begin{definition}\label{def-str-mov}
We call a category $K$ \emph{strongly movable}, if it is movable with respect to itself and the identical functor $1_K$. In other words, if for any object $X\in ob(K)$ there exist an object $M(X)\in ob(K)$ and a morphism   $m_X\in Mor_{K}(M(X), X)$, such that for any object $Y\in ob(K)$ and any morphism $p\in Mor_{K}(Y, X)$ there is a morphism $u_p\in Mor_K(M(X), Y)$ for which $p\circ u_p = m_X$.
\end{definition}

The naturalness of this definition will become evident later, after the proof of Theorem \ref{th-main1}.

\begin{example}
Let $S_0$ be a category of sets with a tagged element, whose objects are all possible pairs $(A,a_0)$, where $A$ is a nonempty set and let $a_0$ be some tagged element of $A$. Further, let the morphisms of $S_0$ be all possible mappings $(A,a_0)\to (B,b_0)$ transforming $a_0$ to a tagged element $b_0$. Then it is not difficult to verify that the category $S_0$ is strongly movable.
\end{example}

\begin{example}
Let $Gr$ be a category of groups whose objects are all possible groups and morphisms are all possible homomorphisms between groups. Then, one can easily verify that the category $Gr$ is strongly movable.
\end{example}

The next theorem gives many interesting, non trivial examples of strongly movable categories.

\begin{theorem}\label{lem1}
Let $Q$ be an arbitrary $CW$-complex. Then the comma category $W^Q$ is strongly movable.
\end{theorem}

\begin{proof}
Assuming that  $X=f':Q\to Q'$ is any object of the category $W^Q$, we will prove that the object $M(X)=1_Q: Q\to Q$ and the morphism $m_X=f': M(X)\to X$ of the comma category $W^Q$
satisfy the conditions of strong movability for the comma category $W^Q$ (see Definition \ref{def-str-mov}). Indeed, let $Y=f'':Q\to Q''$ be an arbitrary object and let $p:Y\to X$ be an arbitrary morphism of the comma category $W^Q$. By the definition of the comma category $W^Q$, $p=\eta:Y\to X$ is a homotopy class satisfying the equality
\begin{equation}\label{eq8}
\eta \circ f'' = f'.
\end{equation}
We consider $u_p:M(X)\to Y$ as the required morphism $u_p=f'':M(X)\to Y$. Then, it remains to verify that $p \circ u_p = m_X$, which coincides with the equality \ref{eq8} accurate within notation.
\end{proof}

\begin{proposition}
If $K$ is a strong movable category, then it is movable with respect to any category $L$ and any functor $\Phi:K\to L$.
\end{proposition}

\begin{proof}
By Definition \ref{def-str-mov} of strong movability, $K$ is movable with respect to itself and the identical functor $1_K:K\to K$. Hence, by Proposition \ref{prop1} the category $K$ is movable with respect to the category $L$ and the functor $\Phi \circ 1_K = \Phi :K\to L$.
\end{proof}

\begin{proposition}
Let a category $K$ be movable with respect to the category $L$ and the functor $\Phi:K\to L$. If $\Phi:K\to L$ is functorial domination, then $K$ is a strongly movable category.
\end{proposition}

\begin{proof}
In our conditions, there exists a functor $\Psi:L\to K$ such that $\Psi \circ \Phi = 1_K$. Hence, by Proposition \ref{prop1} the category $K$ is movable with respect to itself and the functor $\Psi \circ \Phi = 1_K:K\to K$. Thus, $K$ is a strongly movable category.
\end{proof}

\begin{theorem}\label{th-proizved}
The product $\prod\limits_{i\in I}K_i$ of categories $K_i$, $i\in I$, is strongly movable if and only if all its multipliers $K_i$, $i\in I$, are strongly movable.
\end{theorem}

\begin{proof}
Let the product $\prod\limits_{i\in I}K_i$ be strongly movable. For proving that for any $i_0\in I$ the category $K_{i_0}$ is strongly movable, we suppose that $X$ is an arbitrary object of the category $K_{i_0}$ and $(X_i)\in \prod\limits_{i\in I}K_i$ is an object, where $X_{i_0}=X$. Then, by strong movability of the product $\prod\limits_{i\in I}K_i$, for $(X_i)\in \prod\limits_{i\in I}K_i$ there exist an object $(M(X_i))\in \prod\limits_{i\in I}K_i$ and a morphism $(m_{X_i}):(M(X_i))\to (X_i)$. By $M(X_{i_0})$ we denote the $i_0$-th component of the object $(M(X_i))$ and by
$m_{X_{i_0}}$ the $i_0$-th component of the morphism $(m_{X_i})$.

Let now $Y$ be an object and $p:Y\to X$ be a morphism of the category $K_{i_0}$. We consider the object $(Y_i)\in \prod\limits_{i\in I}K_i$ and the morphism $(p_i):(Y_i)\to
(X_i)$, where $Y_{i_0}=Y$, $p_{i_0}=p$, and $Y_i=X_i$, $p_i=1_{X_i}$, for $i\neq i_0$. Let $(u_i):(M(X_i))\to (Y_i)$ be morphism whose existence follows from the definition of the strong movability of the object $(X_i)$, i.e. $(p_i)\circ (u_i)=(m_{X_i})$. Then, the morphism $u_{i_0}:M(X_{i_0})\to Y_{i_0}=Y$ satisfies the equality $p_{i_0}\circ u_{i_0}=m_{X_{i_0}}$, and hence $K_{i_0}$ is a strongly movable category.

Now, we prove the converse statement, i.e. if the category $K_i$ is strongly movable for any $i\in I$, then also the product $\prod\limits_{i\in I}K_i$ is a strongly movable category. To this end, we suppose that $(X_i)$ is an object of the category $\prod\limits_{i\in I}K_i$. Then for each $X_i\in K_i$ there is an object $M(X_i)$ and a morphism $m_{X_i}:M(X_i)\to X_i$, whose existence follows from the definition of the strong movability for $K_i$. We
consider the object $(M(X_i))\in \prod\limits_{i\in I}K_i$ and the morphism $(m_{X_i}):(M(X_i))\to (X_i)$.

Let $(Y_i)\in \prod\limits_{i\in I}K_i$ be an object and let $(p_i):(Y_i)\to (X_i)$ be a morphism. Then for each $p_i:Y_i\to X_i$, $i\in I$, there is a morphism $u_i:M(X_i)\to Y_i$ such that $p_i\circ u_i = m_{X_i}$. Hence, the morphism $(u_i): (M(X_i))\to (Y_i)$ satisfies the equality $(p_i)\circ (u_i) = (m_{X_i})$ which means that the category $\prod\limits_{i\in I}K_i$ is
strongly movable.
\end{proof}

\begin{definition}\label{def-sd}
We say that a category $K$ is functorially weak dominated by a category $L$, or briefly write  $K\lesssim L$, if there exist some functors $F:K\to L$ and  $G:L\to K$, such that $G\circ F \rightsquigarrow 1_K$ (see Definition \ref{def-fm}).
\end{definition}

\begin{remark}\label{rem1}
It is obvious that functorial domination implies weak functorial domination, i.e. $K\leqslant L$ implies $K\lesssim L$.
\end{remark}

\begin{theorem}\label{th2}
If $K\lesssim L$, then the strong movability of the category $L$ implies the strong movability of the category $K$.
\end{theorem}

\begin{proof}
We suppose that $F:K\to L$ and $G:L\to K$ are functors, such that $G\circ F\rightsquigarrow 1_K$, and $\varphi : G\circ F\to 1_K$ is some natural transformation. Then, for an arbitrary object $X$ of the category $K$ we consider the object $M(F(X))$ and the morphism $m_{F(X)}:M(F(X))\to F(X)$ of the category $L$, whose existence follows from the definition of strong movability for $L$. Besides, for $X\in K$ the object $M(X)=G(M(F(X)))\in K$ and the morphism $m_{X}=\varphi(X)\circ G(m_{F(X)}):M(X)\to X$ satisfies the conditions of strong movability of the category $K$. Indeed, let $Y$ be an object and let $p:Y\to X$ be a morphism of the category $K$. By movability of the category $L$ for the morphism $F(p):F(Y)\to F(X)$ of $L$ there is a morphism $v:M(F(X))\to F(Y)$ such that $F(p)\circ v=m_{F(X)}$. It remains to see that the morphism $u=\varphi(Y)\circ G(v):M(X)\to Y$ satisfies the equality $p\circ u = m_{X}$. Indeed, 
$$p\circ u = (p\circ \varphi(Y)) \circ G(v)) = (\varphi(X)\circ G(F(p))) \circ G(v)) = \varphi(X)\circ G(m_{F(X)}) = m_{X}.$$
Thus, $K$ is a strongly movable category, and the proof is complete.
\end{proof}

Theorem \ref{th2} and Remark \ref{rem1} immediately imply the following corollary.

\begin{corollary}\label{cor1}
If a category $K$ is functorially dominating the category $L$, i.e. $K\leqslant L$, and the category $L$ is strongly movable, then also $K$ is strongly movable.
\end{corollary}

\section{Strong Movability Of Topological Spaces}

The notion of strong movability of topological spaces was introduced by Marde\v{s}i\'{c} \cite{mardsegal2} (see Definition \ref{def1}). For the aims of this paper, we need to restate this definition in a somewhat different form.

\begin{proposition}\label{pr1}
Let the inverse system $\{X_\alpha , p_{\alpha \alpha '}, A\}$ of a category $HCW$ be associated with a topological space $X$. Then $X$ is strongly movable if and only if the following condition is fulfilled:

(SM2) for any $\alpha \in A$, there is $\alpha ' \in A$, $\alpha ' \geqslant \alpha$, such that for any $\alpha '' \in A$,
$\alpha '' \geqslant \alpha $, there exists a homotopy class $r^{\alpha ' \alpha ''}:X_{\alpha '}\to X_{\alpha ''}$ for which the following equalities simultaneously are true:
\begin{equation}\label{eq3-0}
p_{\alpha \alpha''}r^{\alpha ' \alpha ''} = p_{\alpha \alpha'},
\end{equation}
\begin{equation}\label{eq3}
r^{\alpha ' \alpha ''} \circ p_{\alpha'}=p_{\alpha''}.
\end{equation}
\end{proposition}

\begin{proof}
Let $X$ be a strongly movable space. Then, in accordance with Definition \ref{def1} it suffices to see that the equality \eqref{eq3} is true. This equality follows from formula \eqref{eq1} and condition $1^\circ$ of Definition \ref{defass}, i.e.
$$r^{\alpha ' \alpha ''} \circ p_{\alpha'}=r^{\alpha ' \alpha ''} \circ p_{\alpha'\alpha^*}\circ p_{\alpha^*}=p_{\alpha''\alpha^*}\circ p_{\alpha^*}=p_{\alpha''}.$$
To prove that, conversely, the condition (SM2) provides the strong movability of the topological space $X$, i.e. the condition (SM1) of Definition \ref{def1}, we prove the existence of an index $\alpha ^* \in A$, $\alpha ^* \geqslant \alpha '$, $\alpha ^* \geqslant \alpha ''$, for which formula \eqref{eq1} is true. To this end, we observe that for any index $\alpha ^0 \in A$, $\alpha ^0 \geqslant \alpha '$, $\alpha ^0 \geqslant \alpha ''$, the equality \eqref{eq3} can be rewritten in the form
$$r^{\alpha ' \alpha ''} \circ p_{\alpha '\alpha^0} \circ p_{\alpha ^0}=p_{\alpha ''\alpha^0} \circ p_{\alpha ^0}.$$
By the condition $3^\circ$ of Definition \ref{defass} there is an index $\alpha ^* \in A$, $\alpha ^* \geqslant \alpha ^0$, such that 
$$r^{\alpha ' \alpha ''} \circ p_{\alpha '\alpha^0} \circ p_{\alpha ^0\alpha ^*}=p_{\alpha ''\alpha^0} \circ p_{\alpha ^0\alpha ^*}$$ 
or
$$r^{\alpha ' \alpha ''} \circ p_{\alpha '\alpha^*}=p_{\alpha ''\alpha^*}.$$
The proof is complete.
\end{proof}

\begin{theorem}\label{th-main}
A topological space $X$ is strongly movable if and only if the following condition is fulfilled:

$(*)$ for any CW-complex $Q$ and any homotopy class  $f:X\to Q$, there are a CW-complex $Q'$ and homotopy classes $f':X\to Q'$ and $\eta :Q'\to Q$ which satisfy the equality $f=\eta \circ f'$, so that for any CW-complex $Q''$ and any homotopy classes $f'':X\to Q''$ and $\eta ' :Q''\to Q$, for which $f=\eta '\circ f''$, there is a homotopy class $\eta '':Q'\to Q''$ such that
\begin{equation}\label{eq4}
\eta '\circ \eta '' = \eta,
\end{equation}
\begin{equation}\label{eq5}
\eta '' \circ f' =f ''.
\end{equation}
\end{theorem}

\begin{proof}
Supposing that the condition $(*)$ is fulfilled, we consider the spectrum $\{X_\alpha , p_{\alpha \alpha '}, A\}$ associated with $X$ and prove its strong movability, i.e. that it satisfies the condition (SM2) of Proposition \ref{pr1}. To this end, we assume that $\alpha \in A$ is any element and $p_\alpha :X\to X_\alpha $ is the natural projection. Then, for the homotopy class $p_\alpha :X\to X_\alpha $ we consider a suitable CW-complex $Q'$ and the homotopy classes $f':X\to Q'$ and $\eta :Q'\to X_\alpha $ satisfying the condition $\eta \circ f'=p_\alpha $, whose existence follows from $(*)$. Since $assX=\{X_\alpha , p_{\alpha \alpha '}, A\}$, for $f':X\to Q'$ there are $\tilde{a} \in A$, $\tilde{a}
\geqslant \alpha $ and $\tilde{f'}:X_{\tilde {\alpha }} \to Q'$, such that
\begin{equation}\label{1}
f'=\tilde{f'}\circ p_{\tilde{\alpha }}.
\end{equation}
It is not difficult to verify that
\begin{equation}\label{2}
p_{\alpha \tilde {\alpha }} \circ p_{\tilde {\alpha }}= \eta
\circ \tilde{f'} \circ p_{\tilde{\alpha }}.
\end{equation}
Indeed,
\[
\eta \circ \tilde{f'} \circ p_{\tilde{\alpha }}=\eta \circ
f'=p_\alpha= p_{\alpha \tilde {\alpha }} \circ p_{\tilde {\alpha
}},
\]
and by the equality \eqref{2} and Definition \ref{defass}, there is an index $\alpha ' \in A$, $\alpha '\geqslant
\tilde{\alpha}$, such that
\begin{equation}\label{3}
p_{\alpha \tilde {\alpha }} \circ p_{\tilde{\alpha
}\alpha '}= \eta \circ \tilde{f'} \circ p_{\tilde{\alpha }\alpha
'}.
\end{equation}
It turns out that this $\alpha ' \in A$ satisfies the condition (SM2) of Proposition \ref{pr1}. Indeed, if $\alpha '' \in A$, $\alpha ''\geqslant \alpha$, is any element, then by $(*)$ for the homotopy classes $p_{\alpha \alpha ''}:X_{\alpha ''}\to X_\alpha$ and $p_{\alpha ''}:X\to X_{\alpha ''}$ there exists a homotopy class $\eta '':Q' \to X_{\alpha ''}$ such that
\begin{equation}\label{4}
\eta=p_{\alpha \alpha ''} \circ \eta '' ,
\end{equation}
\begin{equation}\label{eq6}
\eta '' \circ f' =p_{\alpha ''}.
\end{equation}

Now, it is not difficult to verify that
$$r^{\alpha '\alpha ''}=\eta '' \circ
\tilde{f'} \circ p_{\tilde{\alpha} \alpha '}:X_{\alpha '}\to X_{\alpha ''}$$ 
is the required homotopy class, i.e. the equalities \eqref{eq3-0} and \eqref{eq3} are true (see Proposition \ref{pr1}). Indeed, \eqref{eq3-0} immediately follows by (\ref{4}) and (\ref{3}):
\[
p_{\alpha \alpha ''} \circ r^{\alpha '\alpha ''} =p_{\alpha \alpha ''} \circ \eta ''
\circ \tilde{f'} \circ p_{\tilde{\alpha} \alpha '} =\eta \circ \tilde{f'} \circ
p_{\tilde{\alpha} \alpha '} =p_{\alpha \tilde{\alpha}} \circ
p_{\tilde{\alpha} \alpha '} = p_{\alpha \alpha '}.
\]
and \eqref{eq3} follows from (\ref{1}) and (\ref{eq6}):
\[
r^{\alpha ' \alpha ''} \circ p_{\alpha'}=\eta '' \circ
\tilde{f'} \circ p_{\tilde{\alpha} \alpha '} \circ p_{\alpha'}=\eta '' \circ
\tilde{f'} \circ p_{\tilde{\alpha}} =\eta '' \circ f' =  p_{\alpha''}.
\]

Conversely, let $X$ be a strong movable topological space and let $\{X_\alpha , p_{\alpha \alpha '}, A\}$ be the associated with $X$ inverse spectrum. To prove $(*)$ we suppose that $f:X \to Q$ is any homotopy class where $Q$ is a CW-complex. The spectrum $\{X_\alpha , p_{\alpha \alpha '}, A\}$ is associated with the space $X$, and hence there are an index $\alpha \in A$ and a homotopy class $f_\alpha
:X_\alpha \to Q$, such that
\begin{equation}\label{6}
f=f_\alpha \circ p_\alpha.
\end{equation}
For this $\alpha \in A$, there is an index $\alpha ' \in A$,
$\alpha ' \geqslant \alpha$, which satisfies the strong movability condition (SM2) of the space $X$ (see Proposition \ref{pr1}). The CW-complex $X_{\alpha '}$ and the homotopy classes
\begin{equation}\label{14}
f'=p_{\alpha '} :X \to X_{\alpha '} \quad \text{and} \quad \eta = f_\alpha \circ p_{\alpha \alpha '}:X_{\alpha '} \to Q
\end{equation}
satisfy the condition $(*)$. Indeed, let $Q''$ be a CW-complex and let $f'':X \to Q''$, $\eta ':Q'' \to Q$ be homotopy classes satisfying the condition
\begin{equation}\label{8}
f=\eta ' \circ f''.
\end{equation}
Then for the homotopy class $f'':X \to Q''$ there are an index $\alpha '' \in A$, $\alpha '' \geqslant \alpha$, and a homotopy class $\tilde{f''}:X_{\alpha ''} \to Q''$ such that
\begin{equation}\label{9}
f''=\tilde{f''} \circ p_{\alpha ''}.
\end{equation}
From the equalities \eqref{6}, \eqref{8} and \eqref{9}, it follows that
\[
f_\alpha \circ p_{\alpha \alpha''} \circ p_{\alpha ''}= \eta '
\circ \tilde{f''} \circ p_{\alpha ''}.
\]
Consequently, in accordance with Definition \ref{defass} there is an index $\alpha ''' \in A$, $\alpha ''' \geqslant \alpha ''$, such that 
\begin{equation}\label{10}
f_\alpha \circ p_{\alpha \alpha''} \circ p_{\alpha ''
\alpha '''}= \eta ' \circ \tilde{f''} \circ p_{\alpha '' \alpha
'''}.
\end{equation}
By the strong movability of the space $X$ and Proposition \ref{pr1}, there is a homotopy class $r^{\alpha ' \alpha '''}:X_{\alpha '} \to X_{\alpha '''}$ such that
\begin{equation}\label{11}
p_{\alpha \alpha'''} \circ r^{\alpha ' \alpha '''} = p_{\alpha \alpha '},
\end{equation}
\begin{equation}\label{13}
r^{\alpha ' \alpha '''} \circ p_{\alpha'}=p_{\alpha'''}.
\end{equation}

Therefore, after setting
\begin{equation}\label{12}
\eta ''=\tilde{f''} \circ p_{\alpha '' \alpha '''} \circ
r^{\alpha ' \alpha '''}:X_{\alpha '} \to X_{\alpha '''}.
\end{equation}
it remains to see that the homotopy class $\eta '':X_{\alpha
'} \to Q''$ satisfies the equalities \eqref{eq4} and \eqref{eq5}.

Indeed \eqref{eq4} follows from \eqref{12}, \eqref{10}, \eqref{11} and \eqref{14}:
\begin{multline*}
\eta ' \circ \eta ''=\eta ' \circ \tilde{f''} \circ p_{\alpha ''
\alpha '''} \circ r^{\alpha ' \alpha '''} = f_\alpha \circ p_{\alpha \alpha ''} \circ
p_{\alpha '' \alpha '''} \circ r^{\alpha ' \alpha '''} = \\
=f_\alpha \circ p_{\alpha \alpha '''} \circ r^{\alpha ' \alpha '''} = f_\alpha \circ p_{\alpha \alpha '} = \eta .
\end{multline*}
Besides, \eqref{eq5} follows from \eqref{12}, \eqref{14}, \eqref{13} and \eqref{9}:
\begin{multline*}
\eta '' \circ f' = \tilde{f''} \circ p_{\alpha ''\alpha '''} \circ r^{\alpha ' \alpha '''} \circ p_{\alpha '} = \tilde{f''} \circ p_{\alpha '' \alpha '''} \circ p_{\alpha '''} = \tilde{f''} \circ p_{\alpha ''} = f''.
\end{multline*}
The proof is complete.
\end{proof}

The next theorem comes as a corollary of the above Theorem \ref{th-main}.

\begin{theorem}\label{th-main1}
A topological space $X$ is strongly movable if and only if the comma category $W^X$ is strongly movable.
\end{theorem}

One can see that Theorem \ref{th-main1} makes it possible to apply the results on strong movable categories in investigation of the strong movability property of topological spaces. Indeed, Theorems \ref{lem1} and \ref{th-main1} imply the well-known result that \emph{any CW-complex is a strong movable space}. Besides, the shape domination $sh(X)\leqslant sh(Y)$ means the category domination $W^X\leqslant W^Y$, and therefore Theorem \ref{th-main1} and Corollary \ref{cor1} imply another well-known result: \emph{if $sh(X)\leqslant sh(Y)$, and $Y$ is strongly movable topological spaces, then also the topological space $X$ is strongly movable}.

\begin{proposition}\label{th4}
Let $X$ be a non-connected topological sum of some topological spaces  $X_1$ and $X_2$, i.e. $X=X_1\sqcup X_2$. Then $W^X=W^{X_1\sqcup X_2}\lesssim W^{X_1}\times W^{X_2}$.
\end{proposition}

\begin{proof}
We have to show that there are some functors $F:W^{X_1\sqcup X_2}\to W^{X_1}\times W^{X_2}$ and $G:W^{X_1}\times W^{X_2}\to W^{X_1\sqcup X_2}$, such that $G\circ F \rightsquigarrow 1_{W^{X_1\sqcup X_2}}$ (see Definition \ref{def-sd}).

To this end, we suppose that $f:X_1\sqcup X_2\to Q$ is an object of the category $W^{X_1\sqcup X_2}$ and define the functor $F:W^{X_1\sqcup X_2}\to W^{X_1}\times W^{X_2}$ by the assumption that
$$F(f)=(f|_{X_1}:X_1\to Q, f|_{X_2}:X_2\to Q)\in W^{X_1}\times W^{X_2}.$$
Now, supposing that $(f_1:X_1\to Q, f_2:X_2\to Q')$ is any object of the category$W^{X_1}\times W^{X_2}$, we define the functor $G:W^{X_1}\times W^{X_2}\to
W^{X_1\sqcup X_2}$ by assuming that $G(f_1,f_2)=f:X_1\sqcup X_2 \to Q\sqcup Q'$, where $f|_{X_1}=f_1$ and $f|_{X_2}=f_2$. Then it is not difficult to verify that $G\circ F \rightsquigarrow 1_{W^{X\sqcup Y}}$. The proof is complete.
\end{proof}

Theorems \ref{th-proizved}, \ref{th2}, \ref{th-main1} and Proposition \ref{th4} imply the following statement.

\begin{theorem}
If a topological space $X$ has a finite set of connectivity components which all are strongly movable, then $X$ is strongly movable.
\end{theorem}

\end{document}